\documentclass[a4paper,12pt]{article}
\usepackage[utf8]{inputenc}
\usepackage[english]{babel}
\usepackage[margin=60pt]{geometry}

\usepackage{amsthm}
\usepackage{amsmath}
\usepackage{amsfonts}
\usepackage{mathtools} 			
\usepackage[pdfpagelabels]{hyperref}
\usepackage{mathrsfs}
\usepackage{bbm}
\usepackage{stmaryrd}
\usepackage{indentfirst}
\usepackage{thmtools}
\usepackage{thm-restate}
\mathtoolsset{showonlyrefs}

\numberwithin{equation}{section}

\newtheorem{theorem}{Theorem}[section]
\newtheorem{definition}[theorem]{Definition}
\newtheorem{proposition}[theorem]{Proposition}
\newtheorem{corollary}[theorem]{Corollary}
\newtheorem{lemma}[theorem]{Lemma}

\newtheorem{examples}[theorem]{Examples}
\newtheorem{remark}[theorem]{Remark}
\newtheorem{assumption}[theorem]{Assumption}

\usepackage{xcolor}

\newcommand{\nn}{\nonumber}
\DeclareMathOperator*{\argmin}{arg\,min}
\DeclareMathOperator*{\argmax}{arg\,max}
\newcommand{\dd}{\mathrm{d}}
\DeclareMathOperator*{\ncond}{\textnormal{\textsc{Ncond}}}
\newcommand\suite[1]{\left\{#1\right\}_{t\in\N_0}}

\newcommand\rhomin{r^*}
\newcommand\pepe{\qquad\quad}

\newcommand{\N}{\mathbb N}

\newcommand{\E}{\mathbf{E}}
\newcommand{\G}{\mathbf{G}}
\newcommand{\M}{\mathbf{M}}
\newcommand{\U}{\mathbf{U}}
\newcommand{\V}{\mathbf{V}}
\newcommand{\W}{\mathbf{W}}

\newcommand{\II}{\mathbb{I}_\rho}

\newcommand{\AAA}{\mathcal{A}}
\newcommand{\CCC}{\mathcal{C}}
\newcommand{\EEE}{\mathcal{E}_\rho}
\newcommand{\GGG}{\mathcal{G}_\rho}
\newcommand{\III}{\mathcal{I}}
\newcommand{\SSS}{\mathcal{S}}
\newcommand{\NNN}{\mathcal{N}_\rho}
\newcommand{\HHH}{\mathcal{H}}
\newcommand{\BBB}{\mathcal{B}}

\newcommand{\XX}{\mathscr{X}}
\newcommand{\FF}{\mathscr{F}}

\title{Online matching for the multiclass stochastic block model}
\author{Nahuel Soprano-Loto$^1$ \and  Matthieu Jonckheere$^1$
	\and Pascal Moyal$^2$}
\date{
	$^1$LAAS-CNRS \\ \texttt{\{nahuel.soprano-loto, matthieu.jonckheere\}@laas.fr}\\%
	$^2$IECL/INRIA PASTA \\ \texttt{pascal.moyal@univ-lorraine.fr}\\[2ex]%
}

\begin{document}
	\sloppy	
	
	\maketitle
	
	\begin{abstract}
		\noindent
		We consider the problem of sequential matching in a stochastic block model with several classes of nodes and generic compatibility constraints. When the probabilities of connections do not scale with the size of the graph, we show that under the NCOND condition, a simple max-weight type policy allows to attain an asymptotically perfect matching while no sequential algorithm attain perfect matching otherwise.  
		The proof relies on a specific Markovian representation of the dynamics associated with Lyapunov techniques.
		\\
		
	\end{abstract}
	
	%	\tableofcontents

	\section{Introduction}

	A matching in a graph is a subset of edges where no two edges share a common vertex. Finding algorithms to determine matchings of maximum cardinality is an important problem in computer science and combinatorial optimization (see \cite{schrijver2003} and references therein). 
	
	Recently, online matching algorithms have gained significant attention because of their wide range of applications in fields such as job markets, advertisements, dates, and transportation businesses. Unlike traditional matching algorithms, which operate on a fixed input graph, online matching algorithms must handle sequentially arriving nodes and return a matching on-the-fly at each step. This added complexity presents a significant challenge for algorithm design, since the algorithm must make decisions with incomplete information and a limited view of the overall problem.

	%A matching in a graph is a subset of edges where no two edges share a common vertex.
	%Developing algorithms to find a matching of maximum cardinality is an important problem in computer science \cite{LLRKS1986,GI1989,CLRS2009} and combinatorial optimization \cite{PS1982,schrijver2003,LP2009}.
	
	%\marca{Online matching algorithms}

	%\marca{The work of [KVV]}
	
	In their seminal paper, Karp, Vazirani, and Vazirani \cite{karp1990} investigate randomized online matching algorithms for certain classes of non-random graphs. Specifically, they focus on graphs with one-sided bipartite vertex arrivals.
	The authors use the competitive ratio between the worst online and best offline scenarios as the performance metric for evaluating a policy,
	where an offline algorithm is one that can wait until all requests are received before determining its responses.
	The authors establish that a specific randomization strategy can achieve an asymptotically optimal value of order $n(1-e^{-1})$ for the number of online matchings in a bipartite graph of size $n$.

	More recently, several researchers (see \cite{Kerimov2021DynamicMC} for instance) have focused on online algorithms without knowledge of arrival parameters, showing that structured policies, 
	close to asymptotically optimal (built at a fluid scale) behave 
	well. This is one of the motivation for the present work.
	
	%\marca{Matching algorithms in random graphs}
	
	On the other end of the spectrum, there has been an important effort of research on characterizing the so-called matching number (the cardinality of a maximum matching) of random graphs. 
	Matching numbers on sparse configuration models are studied in detail in \cite{bordenave2013}, generalizing the well known results of Karp and Sipzer (see \cite{karp81} and \cite{Aronson1998MaximumMI}).
	Randomized online matching algorithms and their properties in terms of competitive ratio have been studied in \cite{cohen2018} for regular graphs, and very recently for more general sparse graphs in \cite{pascal} and \cite{noiry2021online}.

	To the best of our knowledge, there are no results on multi-class matching on random graphs, neither from the perspective of matching number characterization and 
	from characterizing the performance of online algorithms, though it is of course a crucial model for applications.

	%\marca{Description of our model, our results and our techniques}
	In this study, we address an online matching problem within a stochastic block model consisting of $n$ nodes subject to per-class compatibility constraints.
	More precisely, a sequence of random graphs is constructed on-the-fly in the sense that nodes arrive one-by-one, and their adjacencies are determined upon arrival.
	This can be seen as an exploration of the large graph of size $ n $.
	A random class is assigned to every node according to a fixed distribution over a (finite) set of classes,
	and each node of class $i$ connects with a node of class $j$ with probability $0 \le \rho(i,j) \le 1$ (that do not depend on $n$).
	The matchings are online in the sense that at most one edge  can be added to the matching per time unit.
	
	To tackle this problem, we introduce a family of matching policies based on simple monotonic rules, which extend the well-known max-weight policies. 
	%The matching is online as the nodes
	%can perform matching only at moments of arrivals which are uniform.
	%We define a family of matching policies based on 
	%simple monotonic rules, generalizing the so-called max-weight policies.
	We then use a joint construction of the stochastic block model and the matching process to reach a Markovian representation of the
	number of unmatched nodes.
	This representation combined with
	the specification of matching policies allows us to use 
	Lyapunov techniques by establishing a series of drift inequalities. To achieve this, we build on the techniques introduced in \cite{SS2022}.
	This is in turn allows us to prove stability properties for the underlying Markov process.
	
	This yields important results on the asymptotic optimality of online matching algorithms under specific matching policies. We show that, when a certain condition (known as NCOND in the literature \cite{MM2016}) on the probabilities  $\rho(i,j)$'s is satisfied, the policies that we study are able to achieve asymptotically optimal $O(n)$ matchings.
	Conversely, we prove that no algorithm can stabilize the process formed by unmatched nodes if the $ \ncond $ condition is not met,
	providing in this way a characterization of the maximum stability region.
	These findings reveal a phase transition in the difficulty of the matching problem in a multi-class setting, and offer valuable insights into the fundamental limits of online matching competitive ratios in such contexts. While our study does not address the case of sparse random graphs, this remains an open question for future research.

	\paragraph{Paper structure.}
	Section \ref{sec:preliminaries} is dedicated to presenting the necessary background information prior to the main results.
	It is divided in three subsections.
	In the first one, the construction of the online stochastic block model is given, together with the online matching algorithm.
	In the second subsection, we introduce the condition $ \ncond $ and state the first result that says that we cannot expect to have an asymptotically optimal online matching policy under the lack of this condition.
	In the third subsection, we present the max-weight type policy that concerns us, and give the Markovian representation of the process of unmatched nodes.
	
	Section \ref{sec:main-result} is dedicated to presenting the main results of the article.
	In Subsection \ref{sec:fundamental-drift}, Theorem \ref{env:theorem-main} is highlighted as the principal result, providing a key drift inequality for a simple quadratic function. The remaining portion of the section is devoted to exploring the corollaries that follow from this theorem. Subsection \ref{sec:stability-and-tightness} presents a result concerning stability and tightness, while Subsection \ref{sec:perfect-matching-io} deals with the possibility of a perfect matching occurring infinitely often. Subsection \ref{sec:ergodicity} focuses on ergodicity and asymptotic optimality, while Subsection \ref{sec:long-run-behaviour} discusses long-term behavior.

	In Section \ref{sec:proof-strategy}, the focus is on presenting the proofs, while the proofs of certain technical results needed for deriving the primary drift inequality are deferred to the Appendix \ref{sctn:proofs}.

	\section{Preliminaries}\label{sec:preliminaries}

	\subsection{Online matching on the stochastic block model}\label{sec:model_definition}

	In the online version of the stochastic block model,
	nodes arrive one-by-one, and the graph is constructed on the fly, generating an increasing sequence of random graphs $\suite{\G_t}=\suite{(\V_t,\E_t)}$.
	Its construction requires the following ingredients: a finite set of classes $\CCC$, a probability $\nu$ defined on $\CCC$, and a symmetric matrix $\rho=(\rho(i,j))_{i,j\in \CCC}\in[0,1]^{\CCC\times \CCC}$.
	Without loss of generality, we assume that $\nu$ is a positive probability, i.e. $\nu(i)>0$ for every $i\in \CCC$. If this were not the case, we could simply restrict to the support of $ \nu $, i.e. to the subset $ \{i\in\CCC:\nu(i)>0\} $. 
	
	The sequence of graphs is defined inductively, together with a sequence of associated matchings.
	Precisely, let
	$\G_0=(\V_0,\E_0)=(\emptyset,\emptyset)$
	and $\M_0=\emptyset$.
	Also, for any $i\in \CCC$, 
	we introduce the auxiliary elements
	$\V^i_0=\emptyset$ and $\U^i_0=\emptyset$. 
	We interpret, at time $ t $,
	the set $\M_t$ as the matching in the graph $ \G_t $ ($ \M_t $ is a subset of $ \E_t $ containing only edges which do not share endpoints),
	the set $\V^i_{t}$ as the set of all nodes of $\V_{t}$ having class $i$,
	and the set $\U^i_{t}$ as the subset of unmatched nodes of class $i$.
	By unmatched node we mean that the node is not the endpoint of any edge in the matching.
	The class of a node $ v\in \V_t $ will be denoted by $ c(v)\in\CCC $, 
	so $ v\in\V_t^{c(v)} $.
	We will say in this case that $ v $ is a \textit{$ c(v) $-node}.
	Also, for a class subset
	$ \AAA\subset \CCC $,
	an $ \AAA $-node is a node $ v $ such that $ c(v)\in\AAA $.
	
	Suppose that we have defined $\G_{t-1}=(\V_{t-1},\E_{t-1})$, $\M_{t-1}\subset \E_{t-1}$,
	$\{\V_{t-1}^i\}_{i\in \CCC}$ and $\{\U_{t-1}^i\}_{i\in \CCC}$.
	Then the corresponding objects at time $t$ are defined as follows:
	\begin{enumerate}
		\item A node $v_t$ arrives, and we set 
		\begin{align}
		\V_{t}=\V_{t-1}\cup \{v_t\}.    
		\end{align}
		We draw the class $c(v_t)\in \CCC$ of node $v_t$  
		from the distribution $\nu$ on $\CCC$, independently of everything else, and we set
		$\V_t^{c(v_t)}=\V_{t-1}^{c(v_t)}\cup\{v_t\}$,
		and $\V_t^{i}=\V_{t-1}^{i}$ for $ i\in \CCC\setminus\{c(v_t)\} $.
		Next we associate to every node $v\in \V_{t-1}$ an (independent of everything else) Bernoulli random variable $\psi_{t,v}$ of parameter $\rho(c(v_t),c(v))$,
		i.e. $\mathbb P(\psi_{t,v}=1)=1-\mathbb P(\psi_{t,v}=0)=\rho(c(v_t),c(v))$,
		and set
		\begin{align}
		\E_t=\E_{t-1}\cup \{\{v_t,v\}:v\in \V_{t-1}\text{ such that }\psi_{t,v}=1\}.
		\end{align}
		
		\item 
		Let now $ \M_t $ be a matching in $ \G_t $  with the following two constraints: $ \M_{t-1}\subset \M_t $ and $ |\M_t\setminus\M_{t-1}| \allowbreak \le 1 $.
		This means that $ \M_t $ is obtained from $ \M_{t-1} $ either by keeping it as it was or by adding one edge in such a way that the resulting family of edges are mutually disjoint.
		Beyond these restrictions, the choice of $ \M_t $ is absolutely arbitrary.
		Nevertheless, in order to have something concrete in mind, the reader can think that $\M_t$ depends on the current and past states of the system, and on an extra independent source of randomness.
		More precisely, we may think of $ \M_t $ as a function depending on
		$ \{(\G_s,\{\V^i_s\}_{i\in\CCC},\M_s)\}_{s\in\{0,\ldots,t-1\}} $, on the class of the incoming node $ c(v_t) $,
		on the Bernoulli random variables $ \{\psi_{t,v}\}_{v\in\V_{t-1}} $, and possibly on an external ---independent of everything else--- source of randomness.
		We are in the following alternative:
		\begin{itemize}
			\item if $\M_t=\M_{t-1}$, set $\U_t^{c(v_t)}=\U_{t-1}^{c(v_t)}\cup\{v_t\}$,
			and
			$\U_{t}^i=\U_{t-1}^i$ for $i\neq c(v_t)$;
			\item if $ \M_t\setminus\M_{t-1} $ has one element, say $\{v,w\}$,
			set 
			$\U_t^{c(v_t)}=(\U_{t-1}^{c(v_t)}\cup\{v_t\})\setminus \{v,w\}$, and $\U^i_t=\U_{t-1}^i\setminus \{v,w\}$ for $i\neq c(v_t)$.
		\end{itemize}
		In the second case in which the edge $ \{v,w\} $ is added to the matching, we say that the nodes $ v $ and $ w $ are \textit{matched}.
		In addition to describing the disjoint set of edges, the word ``matching'' will also be used to refer to a pair of nodes that have been matched.
		Under this terminology, we can say that a matching consists in a family of matchings.
		We hope that this undesirable terminology will not lead to confusion, with the context helping to distinguish the different uses of the word.
	\end{enumerate}

\paragraph{A note on the classical stochastic block model.}
The stochastic block model is a widely-used random graph model for studying clustering and community detection. It can be easily defined as follows: a fixed, finite set of nodes $V$ is divided into $r$ communities $V^1, \ldots, V^r$. For each pair of distinct nodes, a link is created between them in a random fashion, independently of everything else, with a probability that depends only on the communities to which the nodes belong. In the construction of the online stochastic block model, conditioned to having communities $\{\V^i_t\}_{i \in \CCC}$ at time $t$, the distribution of the set of nodes $\E_t$ corresponds to that of the stochastic block model with these communities, and the connection probabilities are given by the matrix $\rho$.

	\subsection{The condition NCOND}\label{sec:ncond}

	The connection probabilities $ (\rho(i,j))_{i,j\in \CCC} $
	define an adjacency matrix $ (\EEE(i,j))_{i,j\in \CCC} $ in the following natural way:
	$\EEE(i,j)=\mathbbm 1\{\rho(i,j)>0\} $ for every $ i,j\in\CCC $.
	If two classes $ i,j\in \CCC $ are such that $ \EEE(i,j)=1 $, we say that they are \textit{compatible}.
	Call $ \GGG=(\CCC,\EEE) $ the graph defined by this adjacency matrix,
	graph which we will refer to as the {\em root graph}.
	Observe that this graph is undirected, in the sense that $ i $ is compatible with $ j $ if and only if $ j $ is compatible with $ i $, and that it admits self-loops, in the sense that there may exist classes that are compatible with themselves.
	A class-subset $ \III\subset \CCC $ is said to be a $\GGG$-independent set if $ i,j\in \III $ implies $ \EEE(i,j)=0 $.
	According to this definition, the empty set is a $ \GGG $-independent set.
	Observe also that if $ \III $ is a $ \GGG $-independent set, then
	the classes lying in $ \III $ cannot have self-loops.
	It is convenient to define the set
	\begin{align}
	\II=\{\III\subset\CCC:\mbox{$ \III $ is a non-empty $ \GGG $-independent set}\}.
	\end{align}
	Also for a class-subset $ \AAA\subset \CCC $, let $\NNN(\AAA)=\{j\in \CCC: \EEE(i,j)=1\text{ for some }i\in \AAA\} $ denote the set containing the classes that are compatible with some class belonging to $ \AAA $.
	For $ i\in \CCC $, we abuse notation by writing $ \NNN(i) $ instead of $ \NNN(\{i\}) $.
	Let 
	\begin{align}
	\eta=\eta(\GGG,\nu)=\min\{\nu(\NNN(\III))-\nu(\III):\III\in\II\}.
	\end{align}
	Of course, for $ \AAA\subset \CCC $ we are writing $ \nu(\AAA) $ to denote $ \sum_{i\in \AAA}\nu(i) $.
	The parameter $ \eta $ is an important one that quantifies how stable the system is.
	The following definition is similar to the one given for instance in \cite{BMMR2021} and \cite{SS2022}.
	
	\begin{definition}\label{def_ncond}
		We say that the pair $(\GGG, \nu) $ satisfies $ \ncond $ if $ \eta>0 $ or, in other words,
		if 
		\begin{align}
		\nu(\III)<\nu(\NNN(\III))\qquad \forall \III\in\II.
		\end{align}
	\end{definition}
	
	We emphasize that the definition of $\ncond$ depends on $\rho$ only trough $\GGG$.
	In other words, the only thing that matters in deciding whether $\ncond$ is satisfied is not the precise values of the entries of the matrix $ \rho $ but whether they are or not zero.

	For $ t\in\N_0(=\{0,1,\ldots\}) $, the vector $ X_t=(X_t(i))_{i\in\CCC} $ defined as $ X_t(i)= |\U^i_t|$, $ i\in\CCC $, encodes the number of unmatched nodes in each class, and will be the object to be examined in order to determine the quality of a matching algorithm.
	We will refer to the stochastic process $ X=\{X_t\}_{t\in\mathbb N_0} $ as the \textit{unmatched process}.
	The following proposition investigates the sequence of random variables $ \{\|X_t\|\}_{t\in\mathbb N_0} $, where $\|x\|=\max_{i\in \CCC}|x(i)| $ stands for the supremum norm of a vector $ x\in\mathbb R^{\CCC} $.
	More precisely, if $\ncond$ is not satisfied, the proposition establishes a criterion that implies that the sequence is not \textit{tight},
	where by tight we mean that
	\begin{align}
	\lim_{\kappa\to\infty}\limsup_{t\in\mathbb N_0}\mathbb P[\|X_t\|>\kappa]=0.
	\end{align}
	This notion of tightness coincides with the tightness of the sequence of distributions $\{\mu_t\}_{t\in\mathbb N_0}$ defined on $ \N_0^\CCC $ as $ \mu_t(\cdot)=\mathbb P(X_t\in\cdot) $,
	according to which the sequence  $\{\mu_t\}_{t\in\mathbb N_0}$ is said to be \textit{tight} if for every $ \varepsilon >0 $ there exists a finite set $ \FF=\FF(\varepsilon)\subset \N_0^\CCC $ such that inequality $ \mu_t(\FF)\ge 1-\varepsilon $ holds for every $ t\in\mathbb N_0 $.

	\begin{proposition}\label{env:proposition-necessity}
		If $ \ncond $ is not satisfied, then there exists $ c=c(\nu,\rho)>0 $ 
		such that
		\begin{align}\label{eq:mango}
		\liminf_{t\to\infty}\mathbb P\Big[\|X_t\|\ge \sqrt{t}\Big]\ge c.
		\end{align}
		In particular, the sequence of random variables $ (\|X_t\|)_{t\in\mathbb N_0} $ is not tight.
		Moreover,
		if $ \eta<0 $,
		\begin{align}\label{eq:vaso}
		\liminf_{t\to\infty}\frac{\|X_t\|}{t}>0\qquad a.s.
		\end{align}
	\end{proposition}
	
	Observe that requirement $ \eta<0 $ is stronger than the lack of $ \ncond $.
	Equation \eqref{eq:vaso}
	is to be understood as meaning that, under this requirement, it is not possible to get an online matching algorithm producing matchings with sizes in the order of the size of the graph.
	We will see later that, under $ \ncond $,
	this is not anymore the case.
	If $ \eta=0 $, our criteria does not give information in this regard.
	
	\subsection{Matching policies and Markov representation}\label{sec:markov-representation}

	We define now a particular kind of online matching policy.
	Suppose that $ \V_{t-1} $, $ \{\V^i_{t-1}\}_{i\in\CCC} $, $ \E_{t-1} $, $ \M_{t-1} $ and $ \{\U^i_{t-1}\}_{i\in\CCC} $ have been defined,
	and that a node $ v_t $ arrives (with class $ c(v_t) $ drawn using $ \nu $).
	A class $ i^*_t=\phi(X_{t-1},c(v_t)) $ is chosen as a function of the vector $ X_{t-1} $ and the class $ c(v_t) $,
	where $ \phi:\mathbb N_0^\CCC\times \CCC\to \CCC $ is a function to which constraints will be assigned later, but which, for the purpose of analysing the Markovianity of the algorithm, can be considered fully general.
	The next step is to construct $ \E_t $ by adding to $ \E_{t-1} $
	edges of the form $ \{v_t,v\} $ with $ v\in\V_{t-1} $
	(through the random variables $ \{\psi_{t,v}\}_{v\in\V_{t-1}} $).
	We are in the following alternative:
	\begin{itemize}
		\item If any of these new edges has an endpoint in $ \U_{t-1}^{i_t^*} $, i.e. if the edge-set $\tilde{\E}_t= \{e\in \E_t:e\cap\U_{t-1}^{i_t^*}\neq\emptyset\text{ and }e\cap \{v_t\}\neq\emptyset\} $ is non-empty,
		then we define $ \M_t $ from $ \M_{t-1} $ by adding any of the edges from $ \tilde{\E}_t $, it does not matter which one.
		Observe that, in this case, we have $ X_t(i_t^*)=X_{t-1}(i_t^*)-1 $.
		\item If otherwise $ \tilde{\E}_t= \emptyset $, then we set $ \M_t=\M_{t-1} $.
	\end{itemize}
	We note that, in the second case,
	an edge is not added to the matching even
	in the presence of other nodes available to be matched with $ v_t $, i.e. even if there existed a class $ i\neq i^*_t $ and a node $ v\in\U^i_{t-1} $ with $ \{v_t,v\}\in\E_t $.
	
	It is convenient to give an alternative, equivalent way of thinking this algorithm.
	As before, suppose that the elements at time $ t-1 $ have been defined, that a node $ v_t $ has arrived, and that the class $ i_t^* $ has been chosen as a function of $ X_{t-1} $ and $ c(v_t) $.
	If $ \U_{t-1}^{i_t^*}=\emptyset $, then we set $ \M_t=\M_{t-1} $.
	Else, consider an arbitrary enumeration $ \U_{t-1}^{i_t^*}=\{w_1,\ldots,w_L\} $.
	If $ \{v_t,w_1\}\in\E_{t} $, i.e. if $ \psi_{t,w_1}=1 $,
	we set $ \M_t=\M_{t-1}\cup\{\{v_t,w_1\}\} $.
	If otherwise $ \{v_t,w_1\}\notin\E_{t} $,
	we ``try'' with $ w_2 $, this meaning that,
	if $ \psi_{t,w_2}=1 $,
	we set $ \M_t=\M_{t-1}\cup\{\{v_t,w_2\}\} $.
	In case also $ \{v_t,w_2\}\notin \E_t $, we try with $ w_3 $, and so on.
	If we do not succeed with any of the vertices in $ \U_{t-1}^{i_t^*} $, that is, if $ \psi_{t,w_1}=\ldots=\psi_{t,w_L}=0 $, then we set $ \M_t=\M_{t-1} $, no matter if there were more edges in $ \E_t $ available to be added to $ \M_{t-1} $.
	Note that, to define this algorithm, not all the information about $ \E_t $ was used but only the one regarding the Bernoulli's random variables until the first success,
	that is, until the first index $ l\in\{1,\ldots,L\} $ for which $ \psi_{t,l}=1 $.
	This observation could be translated into an inexpensive implementation of the algorithm.

	Under this policy, the stochastic process 
	$\suite{X_t}$
	is a discrete-time Markov chain (DTMC) on the state-space $\N_0^{\CCC}$,
	about which we will set out its transition probabilities below.
	For every $i\in \CCC$, we set $ 1_i\in \N_0^{\CCC} $ to be the canonical vector defined as $ 1_i(j)=\mathbbm 1\{i=j\} $.
	Suppose that, at step $t-1$, the state of the DTMC is $x\in\N_0^{\CCC}$,
	namely $ X_{t-1}=x $.
	A new node $v_t$ is added to the graph, whose class $ i=c(v_t)\in\CCC $ is drawn from the distribution $\nu$ independently of everything else,
	and then another class $j=i_t^*=\phi(x,i)\in\CCC$ is chosen.
	We stress that nothing forbids $i$ and $j$ to be the same class {\em a priori}.
	Next a node $u_t\in \U^j_{t-1}$ is possibly matched with $v_t$. 
	The non-zero transition probabilities of the DTMC determined by this dynamic are given by
	\begin{align}\label{transitions_1}
	P(x,x+1_i)&
	=\nu(i)(1-\rho(i,\phi(x,i)))^{x(\phi(x,i))},\quad x\in\mathbb N_0^\CCC,\\
	\label{transitions_2}
	P(x,x-1_j)&=\sum_{i:\phi(x,i)=j}\nu(i)[1-(1-\rho(i,j))^{x(j)}],\quad x\in\N_0^\CCC:  x-1_j\in \mathbb N_0^\CCC .
	\end{align}
	The transition $ P(x,x+1_i) $ corresponds to the addition of one unmatched $ i $-node.
	For this to happen, there has to be an arrival of an $ i $-node which is not matched with anyone.
	This lack of matching occurs when no  $ \phi(x,i) $-node that was unmatched at time $ t-1 $ resulted adjacent to the arriving node in the construction of $ \E_t $,
	event that has probability $ (1-\rho(i,\phi(x,i)))^{x(\phi(x,i))} $
	(when $ \U_{t-1}^j\neq\emptyset $, this corresponds to the event $ \psi_{t,w_1}=\ldots=\psi_{t,w_L}=0 $ in the alternative construction given in the previous paragraph).
	This explains the r.h.s. of \eqref{transitions_1}.
	We note that this probability is $ 1 $ even in the case in which $ \rho(i,\phi(x,i))=1 $ and $ x(\phi(x,i))=0 $,  in concordance with the convention $ 0^0=1 $.
	The transition $ P(x,x-1_j) $ corresponds to the departure of a $ j $-node.
	For this to occur, there are two necessary conditions. The first one is the arrival of a node whose class $ i $ is such that the class chosen by the policy is $ j $, which provides an explanation for the sum in the r.h.s. of \eqref{2nd_condition}.
	The second is that the matching actually takes place, an event whose complement has probability $ (1-\rho(i,j))^{x(j)} $, making sense of the factors $ 1-(1-\rho(i,j))^{x(j)} $ within the sum.

	\paragraph{Hypotheses over $ \phi $.}

	Let $ \alpha:\CCC\to \{1,\ldots,|\CCC|\} $ be a arbitrary bijective function that establishes a priority between classes.
	The policy $ \phi $ is defined as
	\begin{align}\label{policy}
	\phi(x,i)=\argmax_{j\in \CCC}(w(x(j),\rho(i,j)),\alpha(j))
	\end{align}
	with the lexicographical order between ordered pairs.
	In other words, $ \phi(x,i) $ is the class $ j $ that maximizes the weight $ w(x(j),\rho(i,j)) $, being $ \alpha $ the one that decides in case of ties.
	Here $ w $ is an abstract weight-function with the requirements set out below.
	
	\begin{assumption}[Hypotheses over $ w $]\label{assumption}
		The weight-function $ w:\mathbb N_0\times [0,1]\to [0,\infty) $ is assumed to satisfy the following hypotheses:
		\begin{enumerate}
			\item\label{hyp1} $ w(n,r)>0 $ if and only if $ n,r>0 $;
			\item\label{hyp2} $ w $ is coordinate-wise  non-decreasing:
			$ w(n,r)\le w(n',r) $ and $ w(n,r)\le w(n,r') $ for every $ n\le n' $ and every $ r\le r' $;
			\item\label{hyp3} for every $ r\in (0,1] $ there exists $ m(r)\in\mathbb N $ such that inequalities
			$ w(n-1,1)<w(n,r) $ and $ w(n,1)<w(n+1,r) $ hold for every $ n\ge m(r) $.
		\end{enumerate}
	\end{assumption}

	The first hypothesis prevents the selection of a class with which there is no opportunity for a matching to take place (unless all classes meet the same conditions).
	We will expand on this in Remark \ref{env:mate}.
	The second one is somehow natural: on the one hand, the more individuals there are of certain class, the more interested we are in performing a match in order to prevent nodes from accumulation;
	on the other, the larger $ r $ is, the higher the probability that the involved matching attempts will actually take place.
	
	\begin{definition}\label{def_nstar}
		For every $ r\in (0,1] $, let $ n^*(r)\in \mathbb N $ be defined as
		\begin{align}
		n^*(r)=\min\{m\in\mathbb N:  w(n-1,1)<w(n,r) \text{ and } w(n,1)<w(n+1,r) \ \forall n\ge m \}.
		\end{align}
		The set on which the minimum is taken is non-empty due to the third hypothesis in Assumption \ref{assumption}, and hence the definition of $ n^*(r) $ is proper.
	\end{definition}

	\begin{examples}\label{examples}
		Two relevant examples are $ w_1(n,r)=n\mathbbm 1\{r>0\} $ and $ w_2(n,r)=n(1-(1-r)^n) $.
		The weight-function $ w_1 $ defines a `match the longest' type of policy, since it only considers the number of stored nodes and not the connection probabilities.
		About the second one, expression $ w_2(x(j),\rho(i,j))=x(j)(1-(1-\rho(i,j))^{x(j)}) $ is nothing but the number of $ j $-nodes times the probability of performing a match with one of them if the class $ j $ is chosen.
	\end{examples}
	
	\begin{remark}
		The third hypothesis in Assumption \ref{assumption}  says that, when the number of nodes of certain class is large enough, all the relevance of the weight relies on the first coordinate and not on the second one (Lemma \ref{lemma:indistinguishable} sheds light on this direction).
		This makes all such weights equivalent to $ w_1 $ defined in Examples \ref{examples} in saturated regimes.
		Nevertheless, we think that it is valuable to consider general weight-functions depending also on the connection probabilities, such as $ w_2 $, since they may have a better performance in practice.
	\end{remark}

	\begin{remark}\label{env:mate}
		In this remark, we present an important feature of the family of policies under consideration.
		We stand at $ X_{t-1}=x $, and let $ i $ be the class of the incoming node.
		Let $ \SSS_x=\{i\in\CCC:x(i)>0\} $ be the support of $ x $.
		Assume that $ \NNN(i)\cap\SSS_x\neq \emptyset $,
		or, in other words,
		that there are nodes present in the systems that are likely to be connected to the incoming node.
		Then the class $ j=\phi(x,i) $ chosen by the policy lies in $ \NNN(i)\cap\SSS_x $.
		This follows by assumption \ref{hyp1}, from which a class $ j $ is such that $ w(x(j),\rho(i,j))>0 $ precisely if and only if  $ j\in\NNN(i)\cap \SSS_x $.

	\end{remark}
	
	\paragraph{State-space.}
	In order to present things in a more transparent way, we
	prefer to assume that the initial condition of our DTMC $ X $ is the empty state $ 0\in\N_0^\CCC $,
	but of course this is not a real restriction and other initial distributions could be considered.
	Under this assumption, it is natural to define the following space representing the states that are reachable from the initial condition:
	\begin{align}\label{state-space}
	\XX=\{x\in \N_0^\CCC:\mathbb P(X_t=x)>0 \text{ for some }t\in \mathbb N_0 \}.
	\end{align}
	We restrict from now on our DTMC $ X $ to the state-space $ \XX $.
	Irreducibility is a very desirable condition for the state-space in order to stand on some theoretical basis regarding stability.
	Unfortunately this property is not automatic and will be stated as a hypothesis in some of the forthcoming results.
	Nevertheless, we can give conditions that guarantee such irreducibility.

	\begin{restatable}{proposition}{irreducible}
		\label{prop:irreducible}
		Suppose that for every $ i\in \CCC $ there exists $ j\in \CCC $ such that $ \rho(i,j)>0 $ or, in other words, that $ \NNN(i)\neq \emptyset $.
		Then the transition matrix $ (P(x,y))_{x,y\in\XX} $ is irreducible.
	\end{restatable}
	
	The proof of this proposition is given in the Appendix \ref{sctn:proofs}.
	To conclude this section, we observe that
	the hypothesis on which this proposition is based
	actually follows if we assume $ \ncond $.
	Indeed, for the non-empty independent set $ \{i\} $, we have
	$ \nu(i)<\nu(\NNN(i)) $, so $ \nu(\NNN(i))>0 $ and hence $ \NNN(i)\neq 0 $.

	\section{Main result and corollaries}\label{sec:main-result}
	
	\subsection{Fundamental drift inequality}
	\label{sec:fundamental-drift}
	
	For a function $ h:\XX\to [0,\infty)$, we define its drift  at  $ x\in\XX $ as
	\begin{align}\label{eq:drift-definition}
	\dd h(x)=Ph(x)-h(x)=\sum_{y\in\XX}P(x,y)h(y)-h(x).
	\end{align}
	After some necessary preliminaries,
	we will establish an inequality that controls the drift of the quadratic function defined as $ q(x)=\sum_{i\in\CCC}x(i)^2 $.
	
	Let $ \rhomin $ be the minimum positive connection probability
	\begin{align}
	\rhomin=\min\{\rho(i,j):\rho(i,j)>0\},
	\end{align}
	let $ n^*=n^*(\rhomin) $ (recall Definition \ref{def_nstar}), and let
	\begin{align}
	K= \max_{n\in\mathbb N_0}n(1-\rhomin)^n. 
	\end{align}
	For $ x\in\XX $ and $ \AAA\subset \CCC $, let $\|x\|_\AAA=\max_{i\in \AAA}x(i)$ (observe that $ \|x\|_\CCC $ coincides with the supremum norm $ \|x\| $).
	Finally, define the set of classes with and without self-loops as
	\begin{align}
	\CCC_{\rho}^+=\{i\in \CCC: \rho(i,i)>0\}\mbox{ and }\CCC_{\rho}^0=\{i\in \CCC: \rho(i,i)=0\}.
	\end{align}
	
	\begin{theorem}\label{env:theorem-main}
		Let $ \rho\in [0,1]^{\CCC\times \CCC} $ be a  symmetric matrix and $ \nu $ a positive probability defined on $ \CCC $.
		Assume that the pair $ (\GGG,\nu) $ satisfies $ \ncond $, being $ \GGG $ the graph defined in Subsection \ref{sec:ncond}.
		Let $ w:\mathbb N_0\times [0,1]\to[0,\infty) $ be a weight-function satisfying Assumption \ref{assumption},
		and let $ \phi:\XX\times \CCC\to \CCC $ be the associated policy defined in \eqref{policy}.
		Let $ X=\{X_t\}_{t\in\mathbb N_0} $ be the DTMC with state space $ \XX $ defined in \eqref{state-space} and transition probabilities defined in \eqref{transitions_1} and \eqref{transitions_2}.
		Recall the definition of $ \eta $ given in Subsection \ref{sec:ncond}.
		Then the quadratic function $ q:\XX\to\mathbb R $
		defined as $ q(x)=\sum_{i\in \CCC}x(i)^2 $ satisfies
		\begin{align}\label{eq:main_bound}
		\dd q(x)\le
		-2\eta\|x\|_{\CCC_{\rho}^0}-\sum_{i\in \CCC_{\rho}^+:x(i)\ge n^*}2x(i)\nu(i)+1+2n^*+4K(1+|\CCC_{\rho}^+|)
		\end{align}
		for every $ x\in\XX $.
	\end{theorem}

	\subsection{Stability and tightness}\label{sec:stability-and-tightness}
	When the DTMC $ X=\{X_t\}_{t\in\mathbb N_0} $  is irreducible, we say that it is stable (or positive recurrent) if there exists a unique invariant distribution.
	Let  $ \tau_0=\inf\{t\in\mathbb N:X_t=0\} $ be the first positive visit to the empty state.
	Stability is equivalent to the finiteness of the expectation of  $ \tau_0 $,
	and to the tightness of the sequence of distributions $ \{\mu_t\}_{t\in\mathbb N_0} $ (defined in Subsection \ref{sec:ncond}).

	\begin{corollary}\label{env:corollary-stability}
		Assume that the DTMC $ X $  is irreducible.
		Then $ X $ it is stable if and only if condition $ \ncond $  is fulfilled.
	\end{corollary}
	
	Note that the direction of Theorem \ref{env:corollary-stability} that states that stability is a consequence of $ \ncond $  is simply an application of Proposition \ref{env:proposition-necessity} to our particular policy.
	In the context of this proposition, the more general concept of tightness could be defined as a measurement of stability
	even under the lack of Markovianity.
	With this respect, Proposition \ref{env:proposition-necessity} and Theorem \ref{env:corollary-stability} can be interpreted as meaning that $ \ncond $ is the maximum stability region for general online matching policies:
	the absence of $ \ncond $ precludes any stable policy, whereas its presence allows us to find a stable one.

	\subsection{Perfect matching i.o.}
	\label{sec:perfect-matching-io}
	
	Note that, in particular, Theorem \ref{env:corollary-stability} says that, under $ \ncond $, the DTMC $ X $ is recurrent,
	where recurrence can be equivalently defined as $ \mathbb P[\tau_0<\infty]=1 $
	or as
	$\mathbb P[X_t=0\mbox{ infinitely often}]=1$.
	Since the empty state $ 0 $ corresponds to a perfect matching, we can equivalently said that, almost surely, a perfect matching occurs infinitely often.
	The next result concentrate this reasoning,
	and gives the opposite conclusion if $ \eta<0 $.
	Our criteria does not give information if $ \eta=0 $.
	
	\begin{corollary}\label{env:perfect-matching}
		If $ \ncond $ holds, i.e. if $ \eta>0 $, then, almost surely, the online matching is perfect infinitely often.
		If $ \eta<0 $, then the probability of having a perfect matching infinitely often is zero.
	\end{corollary}
	
	\subsection{Ergodicity and asymptotic optimality}
	\label{sec:ergodicity}
	
	By using classical results in conjunction with the fundamental drift inequality stated in Theorem \ref{env:theorem-main}, we can derive the following result.

	\begin{corollary}\label{env:agua}
		Assume $ \ncond $, and let $ \pi $ be the unique invariant distribution associated to the DTMC $ X $.
		Then 
		\begin{align}\label{eq:ergodic}
		\lim_{t\to\infty}\frac{1}{t}\sum_{s=1}^t\|X_s\|=\int \pi(\dd x)\|x\|<\infty\qquad \mbox{a.s.}
		\end{align}
		and, as a consequence,
		\begin{align}\label{eq:pepino}
		\lim_{t\to\infty}\frac{\|X_t\|}{t}=0\qquad \mbox{a.s.}
		\end{align}
	\end{corollary}
	
	The first line \eqref{eq:ergodic} in this corollary is twofold:
	on the one hand, it asserts the finiteness of the invariant mean $ \int \pi(\dd x)\|x\| $; on the other, it establishes the convergence of the ergodic averages.
	A detailed bound for this invariant mean is given in the proof.
	Since $ \|X_t\| $ quantifies the number of unmatched nodes, the second line can \eqref{eq:pepino} can be interpreted an meaning that, almost surely, the size of the matching is of the order of the size of the graph.
	In other words, the algorithm is asymptotically optimal.

	\subsection{Long-run behavior}
	\label{sec:long-run-behaviour}
	
	Assume that $ \ncond $ holds (and hence $ X $ is irreducible due to the paragraph immediately after Proposition \ref{prop:irreducible}).
	Let $ \pi $ be the unique invariant distribution (defined on $ \XX $) whose existence is guaranteed by Theorem \ref{env:corollary-stability}.
	Since our DTMC has period $ 2 $, we cannot expect to have long-run convergence to $ \pi $.
	Nevertheless a periodic version can be formulated.
	For $  l  \in\{0,1\}$, let
	\begin{align}
	\XX_ l =\Big\{x\in\XX:\sum_{i\in\CCC}x(i)= l  \ (\mbox{mod }2)\Big\},
	\end{align}
	the subspace containing the states whose total number of nodes is congruent with $  l  $ modulo $ 2 $,
	and let $ \pi_ l  $ be defined as
	\begin{align}
	\pi_ l (x)=2\pi(x)\mathbbm 1\{x\in\XX_ l \}.
	\end{align}
	For $  l \in\{0,1\} $, $ \pi_ l  $ is in fact a probability due to the ergodic theorem (see for instance Subsection 4.2.1 in \cite{Bre2020}).
	The following result is simply a restatement of Theorem 4.2.3 in \cite{Bre2020}.
	\begin{corollary}
		Under $ \ncond $, the DTMC $ X $ satisfies
		\begin{align}
		\lim_{t\to\infty}\sum_{x\in\XX_{l}}
		|\mathbb P(X_{2t+l}=x)-\pi_l(x)|=0,\qquad l\in\{0,1\}.
		\end{align}
	\end{corollary}

	\section{Proofs}\label{sec:proof-strategy}

	\subsection{Proof of Proposition \ref{env:proposition-necessity}}
	
	We start by establishing a lower bound for the unmatched process.
	For a fixed $ \GGG $-independent set $ \III $,
	and for $ t\in\N_0 $,
	let $ \xi^\III_t=\mathbbm 1\{c(i_t)\in \III\}-\mathbbm 1\{c(i_t)\in \NNN(\III)\} $.
	Observe that $ \{\xi^\III_t\}_{t\in\mathbb N} $ is an i.i.d. sequence of random variables with common distribution $ \mathbb P[\xi^\III_t=1]=\nu(\III) $,
	$ \mathbb P[\xi^\III_t=-1]=\nu(\NNN(\III)) $,
	and $ \mathbb P[\xi^\III_t=0]=1-\mathbb P[\xi^\III_t=1]-\mathbb P[\xi^\III_t=-1] $.
	If we define $ S^\III_t=\sum_{s=1}^t\xi^\III_s $, $ t\in\N_0 $, then $ \{S^\III_t\}_{t\in \mathbb N_0} $ is a lazy random walk on $ \mathbb Z $.
	The key observation is that
	\begin{align}\label{eq:stochastic_bound}
	|\CCC|\|X_t\|\ge \sum_{i\in\III}X_t(i)\ge S_t^\III \qquad \forall t\in\mathbb N_0.
	\end{align}
	This is because the number of unmatched $ \III $-nodes at time $ t $ is equal to the number of $ \III $-node arrivals minus the number of matchings involving $ \III $-nodes, and because all these matchings necessarily involve $ \NNN(\III) $-nodes.
	Assume now that $ \ncond $ does not hold, i.e. $ \eta\le 0 $,
	and let 
	$ \III\in\II $ be such that
	$ \nu(\NNN(\III))-\nu(\III)=\eta $.
	Since $ \mu=-\eta $ and $ \sigma^2=\nu(\III)+\nu(\NNN(\III))>0 $ respectively are  the mean and the variance of $ \xi^\III_t $, we get
	\begin{align}
	\mathbb P\Big[\|X_t\|\ge \sqrt t\Big]&\ge
	\mathbb P\Big[S_t^\III\ge |\CCC| \sqrt t\Big]
	\\
	&=
	\mathbb P\Big[\frac{S_t^\III-\mu t}{ \sqrt t\sigma}\ge \frac{|\CCC|\sqrt t-\mu t}{\sqrt t\sigma}\Big]
	\\
	&\ge
	\mathbb P\Big[\frac{S_t^\III-\mu t}{\sqrt t\sigma}\ge \frac{|\CCC|}{\sigma}\Big]\xrightarrow[t\to\infty]{}1-\Phi\Big(\frac{|\CCC|}{\sigma}\Big)
	\end{align}
	due to the central limit theorem,
	where $ \Phi $ stands for the normal standard cumulative distribution function.
	Take $ c= 1-\Phi(\frac{|\CCC|}{\sigma})$ to get \eqref{eq:mango}.
	
	Assume now that $ \eta<0 $,
	and take $ \III\in\II $ achieving the minimum, that is, such that $ \nu(\NNN(\III))-\nu(\III)=\eta $.
	The strong law of large numbers gives
	\begin{align}
	\lim_{t\to\infty}\frac{S^\III_t}{t}=-\eta>0\qquad a.s.,
	\end{align}
	and
	\eqref{eq:vaso} follows by stochastic comparison.

	\begin{remark}\rm
		The above proof can be reformulated in a slightly different, but insightful way.
		Suppose that $ \eta<0 $,
		let $ \III\in\II $ be a minimizer set,
		and let $ \{S^\III_t\}_{t\in\N_0} $ be its associated lazy random walk.
		Again by the strong law of large numbers,
		\begin{align}
		\frac{1}{t}\bigg[\sum_{i\in\III}|\V_t^i|-\sum_{i\in\NNN(\III)}|\V_t^i|\bigg]
		= \frac{S_t^\III}{t}\xrightarrow[t\to\infty]{}-\eta>0\qquad a.s.
		\end{align}
		In particular, there exists a (random) point in time after which the set $\bigcup_{i\in \III} \V_t^i$ always has 
		strictly more nodes than $\W_t=\bigcup_{i\in \NNN(\III)} \V_t^i$.
		For all such $t$, 
		the set $\bigcup_{i\in I} \V_t^i$ is itself a $ \G_t $-independent set,
		so the subgraph of $ \G_t $ induced by the node-subset
		$	\V_t\setminus\W_t$
		has more than $ \sum_{i\in \III} |\V_t^i| $ odd connected components (every singleton that is a subset of an independent set is an odd connected component).
		Summarizing, $ \W_t $ is a node-subset such that the subgraph in $ \G_t $ induced by its complement has a number of odd connected components that exceeds its cardinality.
		Hence $\W_t$ does not satisfy Tutte's condition, implying that $ \G_t $ cannot have a perfect matching.
	\end{remark}

	\subsection{Proof of Corollary \ref{env:corollary-stability}}

	As already mentioned, one direction of Corollary \ref{env:corollary-stability} follows immediately from Proposition \ref{env:proposition-necessity}, so it only remains to prove that $ \ncond $ implies stability.
	We say that a function $ h:\XX\to[0,\infty) $ is \textit{Lyapunov} if there exist $ \varepsilon >0 $ and $ \FF\subset \XX $ finite such that
	\begin{align}\label{1st_condition}
	Ph(x)=\sum_{y\in\XX}P(x,y)h(y)<\infty\mbox{ for every $ x\in \FF $}
	\end{align}
	and
	\begin{align}\label{2nd_condition}
	\dd h(x)\le-\varepsilon\mbox{ for every $ x\in\XX\setminus \FF $.}
	\end{align}
	Foster-Lyapunov theorem (see Theorem 7.1.1 in \cite{Bre2020} for instance)
	guarantees stability if a Lyapunov function is found.
	In our case, stability will be obtained by proving, as a direct consequence of Theorem \ref{env:theorem-main}, that the quadratic function $ q $ is Lyapunov.
	Since the set $ \{y\in\XX: P(x,y)\neq 0\} $ is finite for every $ x\in \XX $, condition \eqref{1st_condition} is fulfilled automatically.
	The real challenge is condition \eqref{2nd_condition},
	that follows by choosing $ \varepsilon =1 $ and
	\begin{align}
	\FF=\Big\{x\in\XX: \|x\|_{\CCC_{\rho}^0}\le \frac{1+n^*+2K(1+|\CCC_{\rho}^+|)}{\eta},
	\|x\|_{\CCC_{\rho}^+}\le \frac{2+2n^*+4K(1+|\CCC_\rho^+|)}{\min_{i\in \CCC_{\rho}^+}\nu(i)} \Big\},
	\end{align}
	the minimum over an empty set defined as $ +\infty $,
	and by using inequality \eqref{eq:main_bound}.

	\subsection{Proof of Corollary \ref{env:perfect-matching}}

	We give a quick proof of the missing direction of Corollary \ref{env:perfect-matching}.
	Assume $ \eta<0 $, and let $ \III\in \II $ attain the minimum.
	In this case, the associated lazy random walk $ \{S_t^\III\}_{t\in\N_0} $ is transient to infinity, and the result follows by stochastic domination.
	
	\subsection{Proof of Corollary \ref{env:agua}}

	We have
	\begin{align}
	\eta\|x\|_{\CCC_\rho^0}+\sum_{i\in \CCC_\rho^+:x(i)\ge n^*}x(i)\nu(i) \ge (\|x\|-n^*)(\eta\wedge\min_{i\in \CCC_\rho^+}\nu(i)).
	\end{align}
	To see why this it true, separate into the cases $ \|x\|=\|x\|_{\CCC_\rho^0} $, $ \|x\|=\|x\|_{\CCC_\rho^+}\ge n^* $ and
	$ \|x\|=\|x\|_{\CCC_\rho^+}< n^* $.
	Use this inequality in \eqref{eq:main_bound} to get
	\begin{align}\label{mate1}
	\dd q(x)\le-2\|x\|(\eta\wedge\min_{i\in \CCC_\rho^+}\nu(i))+1+2n^*\big[1+(\eta\wedge\min_{i\in \CCC_\rho^+}\nu(i))\big]+4K(1+|\CCC_\rho^+|).
	\end{align}
	Now we can use Theorem 14.3.7 from \cite{MT2009} to obtain
	\begin{align}
	\int\pi(\dd x)\|x\|
	\le
	\frac{1+2n^*\big[1+(\eta\wedge\min_{i\in \CCC_\rho^+}\nu(i))\big]+4K(1+|\CCC_\rho^+|)}{2(\eta\wedge\min_{i\in \CCC_\rho^+}\nu(i))}.
	\end{align}
	Once the finiteness of the invariant expectation has been obtained,
	we can apply Theorem 3.3.2 from \cite{Bre2020} to obtain the convergence of the ergodic averages.
	Finally, equation \eqref{eq:pepino}
	follows because
	\begin{align}
	0\le \frac{\|X_t\|}{t}=\frac{1}{t}\sum_{s=1}^t\|X_s\|
	-\frac{t-1}{t}\frac{1}{t-1}\sum_{s=1}^{t-1}\|X_s\|
	\xrightarrow[t\to\infty]{}0\qquad a.s.
	\end{align}

	\subsection{Proof strategy of Theorem \ref{env:theorem-main}}

	Towards the end of this subsection, we give a short proof of Theorem \ref{env:theorem-main}.
	To arrive at such a proof, we need to prepare the field through a series of previous results, the proof of which are given in the Appendix.
	One of these results, Proposition \ref{prop:lyapunovold},
	is a result that has been obtained in \cite{SS2022} for a more general framework.
	
	The first step is to compare the drift  $ \dd q $ with the drift with respect to another DTMC with homogeneous values of the connection probabilities.
	
	Recall the definition of $ \rhomin $ and, for $ i,j\in \CCC $, let $ \hat\rho $ be the connection probability matrix defined as $ \hat \rho(i,j)=\rhomin\mathbbm 1\{\EEE(i,j)=1\} $.
	Consider the transition probabilities associated to this connection probabilities:
	\begin{align}
	\hat P(x,x+1_i)
	=\nu(i)(1-\hat\rho_{i,\hat\phi(x,i)})^{x(\hat\phi(x,i))}
	\end{align}
	if $ x\in \XX $, and
	\begin{align}
	\hat P(x,x-1_i)=
	\sum_{j:\hat\phi(x,j)=i}\nu(j)[1-(1-\hat\rho(j,i))^{x(i)}]
	\end{align}
	if $ x-1_i\in \XX $,
	with
	\begin{align}
	\hat\phi(x,i)=\argmax_{j\in \CCC}(w(x(j),\hat \rho(i,j)),\alpha(j)).
	\end{align}
	Let $ \hat{\dd} q $ be the generator of $ q $ associated to these transition probabilities, defined as in \eqref{eq:drift-definition} but replacing $ P(x,y) $ with $ \hat P(x,y) $.
	
	\begin{restatable}[Control of the classes without many nodes]{proposition}{homogeneous}
		\label{prop:homogeneous}
		Let $ \varphi:\XX\to\XX $ be the function that conserves only the classes with more than $ n^*$ nodes, defined as
		$(\varphi(x))(i)=x(i)\mathbbm 1\{x(i)\ge n^*\}$.
		Then
		\begin{align}
		\dd q(x)\le \hat{\dd} q(\varphi(x))+2n^*\qquad\mbox{for every $ x\in\XX $.}
		\end{align}
	\end{restatable}
	
	For a class-subset $ \AAA\subset \CCC $, let $ 1_\AAA $ be the configuration defined as $ (1_\AAA)(i)=\mathbbm 1\{i\in \AAA\} $.
	For $ x\in \XX $, 
	the configuration $ x1_\AAA \in \XX $ is defined as the coordinate-wise product, namely $ (x1_\AAA)(i)=x(i)\mathbbm 1\{i\in \AAA\} $.
	Recall the definition of the support of $ x $, $ \SSS_x $,
	given in Remark \ref{env:mate}.
	The following result allows us to deal with the nodes whose classes have self-loops.
	
	\begin{restatable}[Reduction to self-loops-free supports]{proposition}{loops}
		\label{prop:loops}
		If $ x\in \XX $ is such that $ x(i)\ge n^* $ for every $ i\in \SSS_x $, then
		\begin{align}
		\hat{\dd}q(x)\le \hat{\dd}q(x1_{\CCC_\rho^0})+4K|\CCC_\rho^+|-\sum_{i\in \CCC_\rho^+}2x(i)\nu(i).
		\end{align}
	\end{restatable}
	
	The next step reduces the case to configurations whose supports are $ \GGG $-independent sets.
	
	\begin{restatable}[Reduction to support that are $ \GGG $-independent sets]{proposition}{compatible}
		\label{prop:compatible}
		Let $ x\in \XX $ 
		be such that $ \SSS_x\subset\CCC_\rho^0 $,
		and such that $ x(i)\ge n^* $ for every $ i\in \SSS_x $.
		Then
		there exists $ y\in \XX$ such that $ \|x\|=\|y\| $, $ \SSS_y\subset\CCC_\rho^0  $, $ \SSS_y$ is a $ \GGG $-independent set,
		and
		\begin{align}
		\hat{\dd}q(x)\le \hat{\dd}q(y)+2K.
		\end{align}
	\end{restatable}
	
	For $ i,j\in \CCC $, let $ \tilde \rho(i,j)=\EEE(i,j) $, and let $ \tilde{\dd} $ be the associated generator.
	In this dynamic, if the chosen class has nodes, the match is  performed with probability one.
	The next step is to compare the generators $ \hat{\dd} $ and $ \tilde{\dd} $.

	\begin{restatable}[Comparison between $ \hat{\dd} $ and $ \tilde{\dd} $]{proposition}{oldpolicy}
		\label{oldpolicy}
		If  $ x\in\XX $ is such that $ \SSS_x\subset\CCC_\rho^0 $, $ x(i)\ge n^* $ for  $ i\in \SSS_x $, and $ \SSS_x$ is a $ \GGG $ independent set,
		then
		\begin{align}
		\hat{\dd}q(x)\le \tilde{\dd}q(x)+2K.
		\end{align}
	\end{restatable}
	
	The following result is the last one we need.
	Its proof has been given in \cite{SS2022} in a more general context.
	For the benefit of the reader, we present the proof adapted to this case in the Appendix \ref{sctn:proofs}.
	
	\begin{restatable}[Control of $ \tilde d $]{proposition}{lyapunovold}
		\label{prop:lyapunovold}
		Assume that $(\GGG, \nu) $ satisfies $ \ncond $.
		If  $ x\in\XX $ is such that $ \SSS_x\subset \CCC_\rho^0 $ and it is a $ \GGG $-independent set,
		then
		\begin{align}
		\tilde{\dd}q(x)\le 1-2\eta\|x\|.
		\end{align}
	\end{restatable}

	Given all this machinery,
	we can now give a  proof in a few lines of the fundamental drift inequality.
	
	\begin{proof}[Theorem \ref{env:theorem-main}]
		Fix $ x\in \XX $.
		Applying Propositions \ref{prop:homogeneous},
		\ref{prop:loops},
		\ref{prop:compatible},
		\ref{oldpolicy}
		and \ref{prop:lyapunovold},
		we get
		\begin{align}\label{manzana1}
		\dd q(x)&\le \hat{\dd} q(\varphi(x))+2n^*
		\\
		&\le
		\hat{\dd}q(\varphi(x)1_{\CCC_\rho^0})+4K|\CCC_\rho^+|-\sum_{i\in \CCC_\rho^+}2\varphi(x)(i)\nu(i)+2n^*
		\\
		&\le
		\hat{\dd}q(y)+2K+4K|\CCC_\rho^+|-\sum_{i\in \CCC_\rho^+}2\varphi(x)(i)\nu(i)+2n^*
		\\
		&\le
		\tilde{\dd}q(y)+4K+4K|\CCC_\rho^+|-\sum_{i\in \CCC_\rho^+}2\varphi(x)(i)\nu(i)+2n^*
		\\
		&\le 
		1-2\eta\|y\|+4K+4K|\CCC_\rho^+|-\sum_{i\in \CCC_\rho^+}2\varphi(x)(i)\nu(i)+2n^*,
		\end{align}
		where $ y\in\XX $ is such that $ \|y\|=\|\varphi(x)1_{\CCC_\rho^0}\|=\|x\|_{\CCC_\rho^0} $ and $ \SSS_y\in \II $.
		The proof finishes by observing that the last expression coincides with the r.h.s. of \eqref{eq:main_bound}.
	\end{proof}

	\bibliographystyle{plain}
	\bibliography{biblio}

	\appendix
	
	\section{Appendix}\label{sctn:proofs}
	
	\subsection{Proof of Proposition \ref{prop:irreducible}}
	
	We recall the statement for the reader's convenience.
	
	\irreducible*

	By the very definition of $ \XX $, for every $ x\in \XX $ there exists $ t\in \mathbb N_0 $ such that $ P^t(0,x)>0 $ (here $ P^t $ denotes the $ t $-th power of the matrix $ P $).
	It remains to show that we can go from any $ x\in \XX $ to $ 0 $, namely that there exists $ t\in\mathbb N_0 $ such that $ P^t(x,0)>0 $.
	Of course we can assume $ x\neq 0 $, and hence $ \SSS_x\neq \emptyset $.
	The hypothesis of our proposition guarantees that $ \NNN(\SSS_x) \neq \emptyset$,
	so the class of the arriving node, say $ i $, is in $ \NNN(\SSS_x) $ with positive probability.
	Due to Remark \ref{env:mate},
	the chosen class $ j=\phi(x,i) $ lies in $ \NNN(i)\cap \SSS_x $,
	and two nodes are matched with positive probability.
	Summarizing, for every $ x\in \XX\setminus\{0\} $, there exists $ y\in\XX $ such that $ P(x,y)>0 $ and 
	\begin{align}
	\sum_{i\in\CCC}x(i)=1+\sum_{i\in\CCC}y(i).
	\end{align}
	Proceeding in this way finitely many times, we can find $ t\in \N(=\{1,2,\ldots\}) $ such that
	\begin{align}
	P^t(x,0)>0.
	\end{align}

	\subsection{Proof of Proposition \ref{prop:homogeneous}}
	
	We recall its statement:
	
	\homogeneous*

	To prove it, we need to work a bit before.
	We first give an expression for the generator of $ q $.
	First observe that
	\begin{align}
	q(x+1_i)-q(x)=1+2x(i)\text{ for every }x\in\XX
	\end{align}
	and 
	\begin{align}
	q(x-1_i)-q(x)=1-2x(i)\text{ for every }x\in\XX\text{ such that }x-1_i\in\XX.
	\end{align}
	Hence
	\begin{align}
	&\dd q(x)=
	\sum_{y\in \XX}P(x,y)[q(y)-q(x)]
	\\
	&\hphantom{\dd q(x)}
	=\sum_{i\in\CCC}(1+2x(i))P(x,x+1_i)
	+\sum_{i\in \SSS_x}(1-2x(i))P(x,x-1_i)
	\\ \label{expression01}
	&\hphantom{\dd q(x)}
	=1+\sum_{i\in \SSS_x}2x(i)P(x,x+1_i)-\sum_{i\in \SSS_x}2x(i)P(x,x-1_i),
	\end{align}
	where the last identity follows because, after distributing, the factor accompanying the value $ 1 $ represents the total probability of all possible outcomes.
	Replacing by the values of the transition probabilities
	\eqref{transitions_1} and \eqref{transitions_2}, we get
	\begin{align}\label{expression03}
	\dd q(x)&=
	1+\sum_{i\in \SSS_x}2x(i)\nu(i)[1-\rho(i,\phi(x,i))]^{x(\phi(x,i))}
	\\
	&\qquad-\sum_{i\in \SSS_x}2x(i)\sum_{i:\phi(x,i)=j}\nu(i)[1-(1-\rho(i,j))^{x(j)}].
	\end{align}
	Please bear in mind these expression, as they will serve as the foundation for numerous computations.

	We will make use of the following lemma that will be used as follows:
	if the class of the incoming node has at least one neighbor class with more than $ n^* $ nodes, then, for both policies $ \phi $ and $ \hat\phi $, it will necessary chose a class whose number of nodes is maximum.
	This is the main reason for Hypothesis \ref{hyp3} over $ w $ to have been put forward.
	
	\begin{lemma}\label{lemma:indistinguishable}
		Let $ x\in\XX $, and let
		$ i\in\CCC $ be such that $ x(j)\ge n^* $ for some $ j\in \NNN(i) $.
		Then
		\begin{align}
		x(\phi(x,i))=x(\hat\phi(x,i))=\max\{x(j):j\in \NNN(i)\}.
		\end{align}	
	\end{lemma}
	
	\begin{proof}[Proof of Lemma \ref{lemma:indistinguishable}]
		If $ k\in\CCC $ is such that $ x(k)<n^* $, then $ k\neq \phi(x,i) $.
		To see why this is true, take $ j\in \NNN(i) $ such that $ x(j)\ge n^* $, that exists by assumption.
		Then
		\begin{align}
		&w(x(k),\rho(i,k))\le w(n^*-1,\rho(i,k))
		\\
		&\phantom{w(x(k),\rho(i,k))}\le w(n^*-1,1)
		\\
		&\phantom{w(x(k),\rho(i,k))}< w(n^*,\rhomin)
		\\
		&\phantom{w(x(k),\rho(i,k))}\le w(x(j),\rhomin)
		\\
		&\phantom{w(x(k),\rho(i,k))}\le w(x(j),\rho(i,j)),
		\end{align}
		and so the maximum defining $ \phi(x,i) $ cannot be attained at $ k $.
		Hence $ x(\phi(x,i))\ge n^* $.
		Similarly, $ k\notin \NNN(i) $ implies $ k\neq \phi(x,i) $, and hence $ \phi(x,i)\in \NNN(i) $.
		If there existed $ j\in \NNN(i) $ such that   $ x(j)>x(\phi(x,i)) $,
		we would have
		\begin{align}
		&w(x(j),\rho(i,j))\ge w(x(j),\rhomin)
		\\
		&\phantom{w(x(j),\rho(i,j))}> w(x(j)-1,1)
		\\
		&\phantom{w(x(j),\rho(i,j))}\ge  w(x(\phi(x,i)),1)
		\\
		&\phantom{w(x(j),\rho(i,j))}\ge  w(x(\phi(x,i)),\rho(i,\phi(x,i))),
		\end{align}
		contradicting the definition of $ \phi(x,i) $.
		Hence such a $ j $ does not exist, and so
		$ x(\phi(x,i))= \max\{x(j):j\in \NNN(i)\} $ as desired.
		The proof for $ x(\hat\phi(x,i)) $ is analogous.
	\end{proof}

	Now we can move on to the proof of the proposition in issue.
	Fix $ x\in \XX $ and call $ y=\varphi(x)  $.
	Have in mind the following analogous expression of \eqref{expression01} for $ \hat\dd $ instead of $ \dd $ and $ y $ instead of $ x $:
	\begin{align}\label{expression02}
	\hat{\dd}q(y)=1+\sum_{i\in \SSS_y}2y(i)\hat P(y,y+1_i)-\sum_{i\in \SSS_y}2y(i)\hat P(y,y-1_i).
	\end{align}

	We first compare the first sums in \eqref{expression01} and  \eqref{expression02}.
	Let $ \AAA=\{i\in\CCC:  x(i)\allowbreak\ge n^*\} $,
	$ \AAA_1=\{i\in \AAA: \EEE(i,j)=1\text{ for some }j\in \AAA\}=\AAA\cap \NNN(\AAA) $ and $ \AAA_2=\AAA\setminus \AAA_1 $.
	Then
	\begin{align}
	&\sum_{i\in \SSS_x}2x(i) P(x,x+1_i)
	-\sum_{i\in \SSS_y}2y(i)\hat P(y,y+1_i)
	\\
	&\pepe=
	\sum_{i\in \SSS_x\setminus \AAA}2x(i) P(x,x+1_i)
	+\sum_{i\in  \AAA}2x(i) P(x,x+1_i)
	-\sum_{i\in \AAA}2x(i)\hat P(y,y+1_i)
	\\
	&\pepe\le
	2n^*
	+\sum_{i\in  \AAA}2x(i) [P(x,x+1_i)
	-\hat P(y,y+1_i)]
	\\
	&\pepe=
	2n^*
	+\sum_{i\in  \AAA_1}2x(i) \nu(i)\big\{[1-\rho(i,\phi(x,i))]^{x(\phi(x,i))}
	-(1-\rhomin)^{y(\phi(y,i))}\big\}
	\\
	&\pepe\pepe+\sum_{i\in  \AAA_2}2x(i)\nu(i) \big\{[1-\rho(i,\phi(x,i))]^{x(\phi(x,i))}-1\big\}
	\\
	&\pepe\le 2n^*.
	\end{align}
	In the first identity, we used that $ \SSS_y=\AAA $, and that $ x $ and $ y $ coincide there.
	In the last inequality, we used that $ x(\phi(x,i))=y(\phi(y,i)) $ if $ i\in \AAA_1 $ due to Lemma \ref{lemma:indistinguishable}.
	
	We now control the difference between the second sums in \eqref{expression01} and \eqref{expression02}.
	Write the set $ \{x(i):i\in \AAA\}$($=\{y(i):i\in \AAA\} $) as
	$ \{n_1,\ldots,n_L\} $, assuming $ n_l<n_{l+1} $ for every $ l\in\{1,\ldots,L-1\} $, and call $ \HHH_l=\{i\in \AAA:x(i)=n_l\} $ for every $ l\in\{1,\ldots,L\} $.
	The $ \HHH_l $'s are nothing but the level sets of $ x $ (or $ y $) at levels higher or equal than $ n^* $.
	Then
	\begin{align}
	&-\sum_{i\in \SSS_x}2x(i)P(x,x-1_i)
	+\sum_{i\in \SSS_y}2y(i)\hat P(y,y-1_i)
	\\
	&\pepe \le
	\sum_{i\in \AAA}2x(i)[-P(x,x-1_i)+\hat P(y,y-1_i)]
	\\
	&\pepe \le
	\sum_{l=1}^L2n_l\sum_{i\in\HHH_l}[-P(x,x-1_i)+\hat P(y,y-1_i)]
	\\
	&\pepe =
	\sum_{l=1}^L2n_l
	\bigg[
	-\sum_{i\in\HHH_l}\sum_{j:\phi(x,j)=i}\nu(j)(1-(1-\rho(i,j))^{n_l})
	\\
	&\pepe\pepe\pepe+\sum_{i\in\HHH_l}\sum_{j:\hat\phi(x,j)=i}\nu(j)(1-(1-\hat \rho(i,j))^{n_l})
	\bigg]
	\\
	&\pepe \le
	\sum_{l=1}^L2n_l
	\bigg[
	-\sum_{i\in\HHH_l}\sum_{j:\hat\phi(x,j)=i}\nu(j)(1-(1-\hat\rho(j,i))^{n_l})
	\\
	&\pepe\pepe\pepe+\sum_{i\in\HHH_l}\sum_{j:\hat\phi(x,j)=i}\nu(j)(1-(1-\hat\rho(j,i))^{n_l})
	\bigg]
	\\
	&\pepe=0.
	\end{align}	
	The last identity holds because each of the two sums between the square brackets, due again to Lemma \ref{lemma:indistinguishable}, are equal to
	\begin{align}
	(1-(1-\rhomin)^{n_l})\nu\Big(\Big\{ j\in \NNN(\HHH_l):\max_{k\in \NNN(j)}x(k)=n_l\Big\}\Big).
	\end{align}
	Indeed, for the first sum for instance, after a change in the order of the sums,
	we have
	\begin{align}
	\sum_{i\in\HHH_l}\sum_{j:\phi(x,j)=i}\nu(j)(1-(1-\hat\rho(j,i))^{n_l})
	&=\sum_{j:\hat\phi(x,j)\in\HHH_l}\nu(j)(1-(1-\hat\rho(j,\hat\phi(x,j)))^{n_l})
	\\
	&=\sum_{j:\hat\phi(x,j)\in\HHH_l}\nu(j)(1-(1-\rhomin)^{n_l})\mathbbm 1\{j\in \NNN(\HHH_l)\}
	\\
	&=(1-(1-\rhomin)^{n_l})\sum_{j\in \NNN(\HHH_l):\hat\phi(x,j)\in\HHH_l}\nu(j).
	\end{align}

	\subsection{Proof of Proposition \ref{prop:loops}}
	
	We recall the statement:
	
	\loops*
	
	For $ k\in \CCC_\rho^+ $, we will prove that
	\begin{align}
	\hat{\dd}q(x)\le \hat{\dd}q(x1_{C\setminus\{k\}})+4K-2x(k)\nu(k).
	\end{align}
	An iteration of this result let us conclude.
	
	In analogy with expression \eqref{expression03}, we have
	\begin{align}\label{generator_hat}
	\begin{aligned}
	\hat{\dd}q(x)&=1+
	\sum_{i\in \SSS_x}2x(i)\nu(i)[1-\hat\rho(i,\hat\phi(x,i))]^{x(\hat\phi(x,i))}
	\\
	&\pepe-\sum_{i\in \SSS_x}2x(i)\sum_{j:\hat\phi(x,j)=i}\nu(j)(1-(1-\hat\rho(j,i))^{x(i)}).
	\end{aligned}
	\end{align}
	
	We start by controlling the first sum.
	If $ i=k $, we have
	\begin{align}
	2x(k)\nu(k)(1-\hat\rho(k,\hat\phi(x,k)))^{x(\hat\phi(x,k))}
	\le 2x(k)\nu(k)(1-\rhomin)^{x(k)}
	\le 2K,
	\end{align}
	where we have used that $ x(\hat\phi(x,k))\ge x(k) $ due to Lemma \ref{lemma:indistinguishable}.
	For the same reason, if $ i\ne k $,
	\begin{align}
	2x(i)\nu(i)(1-\hat\rho(i,\hat\phi(x,i)))^{x(\hat\phi(x,i))}
	\le 2y(i)\nu(i)(1-\hat\rho(i,\hat\phi(y,i)))^{y(\hat\phi(y,i))},
	\end{align}
	the previous inequality being actually an identity if $ i\notin \NNN(k) $.
	
	To control the second sum, we first observe that, after switching the sums, we have
	\begin{align}
	&\sum_{i\in \SSS_x}2x(i)\sum_{j:\hat\phi(x,j)=i}\nu(j)(1-[1-\hat\rho(j,i)]^{x(i)})
	\\
	&\pepe=
	\sum_{j}
	2x(\hat\phi(x,j))\nu(j)(1-[1-\hat\rho(j,\hat\phi(x,j))]^{x(\hat\phi(x,j))}).
	\end{align}
	For $ j=k $, we have
	\begin{align}
	2x(\hat\phi(x,k))\nu(k)(1-[1-\hat\rho(k,\hat\phi(x,k))]^{x(\hat\phi(x,k))})
	&\ge
	2x(k)\nu(k)(1-(1-\rhomin)^{x(k)})
	\\
	&=2x(k)\nu(k)
	-2x(k)\nu(k)(1-\rhomin)^{x(k)}
	\\
	&\ge 2x(k)\nu(k)
	-2K.
	\end{align}
	If $ j\ne k $, again by Lemma \ref{lemma:indistinguishable},
	\begin{align}
	&2x(\hat\phi(x,j))\nu(j)(1-(1-\hat\rho(j,\hat\phi(x,j)))^{x(\hat\phi(x,j))})
	\\
	&\pepe\ge
	2y(\hat\phi(y,j))\nu(j)(1-(1-\hat\rho(j,\hat\phi(y,j)))^{y(\hat\phi(y,j))}),
	\end{align}
	being an equality if $ j\notin \NNN(k) $.
	
	Putting things together,
	we get
	\begin{align}
	\tilde{\dd}q(x)&=
	1+
	\sum_{i\in \SSS_x}2x(i)\nu(i)(1-\hat\rho(i,\hat\phi(x,i)))^{x(\hat\phi(x,i))}
	\\
	&\pepe-\sum_{j}
	2x(\hat\phi(x,j))\nu(j)(1-(1-\hat\rho(j,\hat\phi(x,j)))^{x(\hat\phi(x,j))})
	\\
	&\le
	1+
	\sum_{i\ne k}2y(i)\nu(i)(1-\hat\rho(i,\hat\phi(y,i)))^{y(\hat\phi(y,i))}
	\\
	&\pepe-\sum_{j\ne k}
	2y(\hat\phi(y,j))\nu(j)(1-(1-\hat\rho(j,\hat\phi(y,j)))^{x(\hat\phi(y	,j))})
	\\
	&\pepe+4K-2x(k)\nu(k)
	\\
	&=
	\tilde{\dd}q(y)+4K-2x(k)\nu(k)
	\end{align}
	as desired.
	
	\subsection{Proof of Proposition \ref{prop:compatible}}
	
	We recall the statement:
	
	\compatible*

	This proposition is obtained by iterating  the following lemma.
	
	\begin{lemma}
		Let $ x\in \XX $ such that $ \SSS_x\subset\CCC_\rho^0 $,
		$ x(i)\ge n^* $ for every $ i\in \SSS_x $,
		and assume that $ \SSS_x $ is not a $ \GGG $-independent set.
		Let $ \AAA $ be a connected component of $ \SSS_x $ with $ |\AAA|>1 $,
		and let
		\begin{align}
		k= \argmin_{i\in \AAA}(x(i),\alpha(i)).
		\end{align} 
		Then
		\begin{align}
		\dd q(x)\le \dd q(x1_{\CCC\setminus\{k\}})+2\nu(k)K.
		\end{align}
	\end{lemma}
	
	\begin{proof}
		Call $ y=x1_{\CCC\setminus\{k\}} $.
		Recall the expression of the generator \eqref{generator_hat}.
		We first dominate the first sum.
		If $ i=k $,
		\begin{align}
		2x(k)\nu(k)(1-\hat\rho(k,\hat\phi(x,k)))^{x(\hat\phi(x,k))} \le 2x(k)\nu(k)
		\le 2\nu(k)K.
		\end{align}
		If $ i\in \SSS_x\setminus\{k\}$ is such that $\hat\phi(x,i)\neq k $, then
		$ \hat P(x,x+1_i)=\hat P(y,y+1_i) $.
		Finally, if $ i\in \SSS_x\setminus\{k\}$ is such that $\hat\phi(x,i)=k$, 
		then $ \NNN(i)\cap \SSS_x=\{k\} $ and hence
		\begin{align}
		2x(i)\nu(i)(1-\hat\rho(i,\hat\phi(x,i)))^{x(\hat\phi(x,i))}\le 2y(i)\nu(i)
		=2y(i)\nu(i)(1-\hat\rho(i,\hat\phi(y,i)))^{y(\hat\phi(y,i))}.
		\end{align}
		Putting these things together, we get
		\begin{align}
		&\sum_{i\in \SSS_x}2x(i)\nu(i)(1-\hat\rho(i,\hat\phi(x,i)))^{x(\hat\phi(x,i))} 
		\\
		&\qquad\le 
		\sum_{i\in \SSS_y}2y(i)\nu(i)(1-\hat\rho(i,\hat\phi(y,i)))^{y(\hat\phi(y,i))}+2\nu(k)K.
		\end{align}
		
		The control of the second sum is more delicate.
		Let $ \BBB_1=\{i\in \NNN(k):\hat\phi(x,i)=k \} $
		and $ \BBB_2=(\NNN(\BBB_1)\cap \SSS_x)\setminus \{k\} $.
		On the one hand, $\hat P(x,x-1_i)=\hat P(y,y-1_i)$ if $ i\in \SSS_x\setminus (\BBB_2\cup\{k\}) $.
		On the other,
		\begin{align}
		&\sum_{i\in \BBB_2}
		2y(i)\hat P(y,y-1_i)
		\\
		&\pepe= \sum_{i\in \BBB_2}
		2y(i)\sum_{j:\hat\phi(y,j)=i}\nu(j)(1-(1-\hat\rho(j,i))^{y(i)})
		\\
		&\pepe=\sum_{j:\hat\phi(y,j)\in \BBB_2}2y(\hat\phi(y,j))\nu(j)(1-(1-\hat\rho(j,\hat\phi(y,j)))^{y(\hat\phi(y,j))})
		\\
		&\pepe= \sum_{j\in \BBB_1^c:\hat\phi(y,j)\in \BBB_2}
		2y(\hat\phi(y,j))\nu(j)(1-(1-\hat\rho(j,\hat\phi(y,j)))^{y(\hat\phi(y,j))})
		\\
		&\pepe\pepe+\sum_{j\in \BBB_1}
		2y(\hat\phi(y,j))\nu(j)(1-(1-\hat\rho(j,\hat\phi(y,j)))^{y(\hat\phi(y,j))})
		\\
		&\pepe\le \sum_{j:\hat\phi(x,j)\in \BBB_2}
		2x(\hat\phi(x,j))\nu(j)(1-(1-\hat\rho(j,\hat\phi(x,j)))^{x(\hat\phi(x,j))})
		\\
		&\pepe\pepe+\sum_{j\in \BBB_1}
		2x(k)\nu(j)(1-(1-\hat\rho(j,k))^{x(k)})
		\\
		&\pepe=\sum_{i\in \BBB_2}2x(i)\hat P(x,x-1_i)+2x(k)\hat P(x,x-1_k).
		\end{align}
		The inequality holds because, if $ j\in \BBB_1 $, then $ \hat\phi(x,j)=k $ and hence $ y(\hat\phi(y,j))\le x(k) $.
		Putting these things together, we get
		\begin{align}
		&-\sum_{i\in \SSS_x}2x(i)\hat P(x,x-1_i)
		\\
		&\qquad= -\sum_{i\in \SSS_x\setminus (\BBB_2\cup\{k\})}2x(i)\hat P(x,x-1_i)
		-\sum_{i\in \BBB_2}2x(i)\hat P(x,x-1_i)
		\\
		&\qquad\qquad-2x(k)\hat P(x,x-1_k)
		\\
		&\qquad\le
		-\sum_{i\in \SSS_y\setminus \BBB_2}2y(i)\hat P(y,y-1_i)
		-\sum_{i\in \BBB_2}2y(i)\hat P(y,y-1_i)
		\\
		&\qquad=
		-\sum_{i\in \SSS_y}2y(i)\hat P(y,y-1_i),
		\end{align}
		which let us conclude.
	\end{proof}

	\subsection{Proof of Proposition \ref{oldpolicy}}

	Next the restatement:
	
	\oldpolicy*
	
	We have
	\begin{align}\label{expression04}
	\hat{\dd}q(x)
	=1+\sum_{i\in \SSS_x}2x(i)\hat P(x,x+1_i)-\sum_{i\in \SSS_x}2x(i)\hat P(x,x-1_i),
	\end{align}
	and the analogous expression for $ \tilde{\dd}q(x) $.
	On the one hand, since $ \SSS_x $ is an independent set, we have
	\begin{align}\label{eq:pipi01}
	\sum_{i\in \SSS_x}2x(i)\hat P(x,x+\eta_i)=
	\sum_{i\in \SSS_x}2x(i)\tilde P(x,x+\eta_i).
	\end{align}
	On the other, for $ n\ge n^* $, Lemma \ref{lemma:indistinguishable} gives
	\begin{align}
	& \sum_{i\in \SSS_x:x(i)=n}2x(i)\tilde P(x,x-\eta_i)
	-\sum_{i\in \SSS_x:x(i)=n}2x(i)\hat P(x,x-\eta_i)
	\\
	&\qquad= \sum_{i\in \SSS_x:x(i)=n}2n
	\sum_{j:\tilde\phi(x,j)=i}\nu(j)
	-\sum_{i\in \SSS_x:x(i)=n}2n
	\sum_{j:\hat\phi(x,j)=i}\nu(j)(1-(1-\hat\rho(j,i))^{n})
	\\
	&\qquad=2n\sum_{j:\|x\|_{\NNN(j)}=n}\nu(j)
	-2n\sum_{j:\|x\|_{\NNN(j)}=n}\nu(j)(1-(1-\rhomin)^{n})
	\\
	&\qquad=2n(1-\rhomin)^n\nu\{j:\|x\|_{\NNN(j)}=n\}
	\\
	&\qquad\le 2K\nu\{j:\|x\|_{\NNN(j)}=n\}.
	\end{align}
	Summing over $ n\ge n^* $, we get
	\begin{align}\label{eq:pipi02}
	\sum_{i\in \SSS_x}2x(i)\tilde P(x,x-\eta_i)
	-\sum_{i\in \SSS_x}2x(i)\hat P(x,x-\eta_i)\le 2K.
	\end{align}
	Facts \eqref{eq:pipi01} and \eqref{eq:pipi02} let us conclude.
	
	\subsection{Proof of Proposition \ref{prop:lyapunovold}}
	
	As usual, we restate the proposition to be proven:
	
	\lyapunovold*

	In this case, $ \tilde{\dd}q(x) $ can we written as
	\begin{align}
	\tilde{\dd}q(x)=1+\sum_{i\in \SSS_x}2x(i)[\nu(i)-\tilde P(x,x-1_i)].
	\end{align}
	
	Since the case where $x=0$ leads to a trivial result, we will proceed by assuming that $ x\neq 0 $.
	Let $ \{a_1,\ldots,a_L\} $ be the image of the function $ \SSS_x\ni i\mapsto x(i) $, assuming $ a_l>a_{l+1} $ for every $ l\in\{1,\ldots,L-1\} $.
	For every $ l\in\{1,\ldots,L\} $,
	let $ \HHH_l=\{i\in \SSS_x:x(i)=a_l\} $.
	We can write
	\begin{align}
	\sum_{i\in \SSS_x}x(i)[\nu(i)-\tilde P(x,x-1_i)]=\sum_{i\in \SSS_x}x(i)r_i
	=\sum_{l=1}^La_l\sum_{i\in\HHH_l}r_i
	=\sum_{l=1}^La_lR_l,
	\end{align}
	where we have defined $ r_i = \nu(i) -\tilde P(x,x-1_i)$ and $ R_l=\sum_{i\in\HHH_l}r_i $.

	If
	$R_l\le 0$ for every $ l\in\{1,\ldots, L\}
	$,
	we are done since
	\begin{align}\label{pera0}
	\sum_{l=1}^L a_l R_l
	\le
	a_1 R_1
	&=
	\| x\|
	\bigg[
	\nu(\HHH_1)-\sum_{i\in\HHH_1}\tilde P(x,x-1_i)
	\bigg]
	\\
	&=
	\| x\|
	[
	\nu(\HHH_1)-\nu(\NNN(\HHH_1))
	]
	\le
	-\| x\|\eta.
	\end{align}
	If not,
	we define
	$L_1= \max\{ l\in\{1,\ldots,L\}:R_l>0 \}$.
	Then
	\begin{align}\label{silla}
	\sum_{l=1}^{L}a_l R_l\le \sum_{l=1}^{L_1}a_l R_l
	\le 
	\sum_{l=1}^{L_1-2}a_{l}R_l+a_{L_1-1}(R_{L_1-1}+R_{L_1}).
	\end{align}
	If $ R_{L_1-1}+R_{L_1}\le 0 $,
	we control the last quantity by $ \sum_{l=1}^{L_1-2}a_{l}R_l $ and restart the procedure. 
	If otherwise $ R_{L_1-1}+R_{L_1}> 0 $,
	we control the r.h.s. of \eqref{silla} by
	\begin{align}
	\sum_{l=1}^{L_1-3}a_{l}R_l+a_{L_1-2}(R_{L_1-2}+R_{L_1-1}+R_{L_1}).
	\end{align}
	We again separate between the cases 
	\begin{align}
	R_{L_1-2}+R_{L_1-1}+R_{L_1}\le 0
	\end{align}
	and 
	\begin{align}
	R_{L_1-2}+R_{L_1-1}+R_{L_1}> 0
	\end{align}
	and proceed as before.
	Finally, if we reach $ \HHH_1 $ with this procedure, we have
	\begin{align}
	a_1\sum_{l=1}^{L_1}R_l
	&
	=\|x\|\bigg[\nu( \HHH_1\cup\ldots \cup \HHH_{L_1} )
	-\sum_{i\in\HHH_1\cup\ldots \cup \HHH_{L_1}}\tilde P(x,x-1_i)\bigg]
	\\ \nn
	&
	=\|x\|[\nu( \HHH_1\cup\ldots \cup \HHH_{L_1} )
	-\nu(\NNN( \HHH_1\cup\ldots \cup \HHH_{L_1}) )]
	\le-\eta\|x\|,
	\end{align}
	which let us conclude.

\end{document}